\documentclass[sn-mathphys,Numbered]{sn-jnl}

\usepackage{amsfonts}
\usepackage{amsmath,amscd,amssymb}
\usepackage{mathrsfs}
\usepackage{amsthm,xcolor}
\usepackage{mathtools}
\usepackage{graphics,graphicx}
\usepackage{tikz}
\usepackage{tikz-cd}
\usetikzlibrary{shapes.geometric, positioning, calc}
\usetikzlibrary{decorations.markings}
\usetikzlibrary{arrows.meta, quotes}
\usetikzlibrary{arrows}
\usetikzlibrary{graphs}
\usepackage{multirow}%
\usepackage{mathrsfs}%
\usepackage[T1]{fontenc}
\usepackage{mathpazo}
\usepackage[title]{appendix}%
\usepackage{textcomp}%
\usepackage{manyfoot}%
\usepackage{booktabs}%
\usepackage{algorithm}%
\usepackage{algorithmicx}%
\usepackage{algpseudocode}%
\usepackage{listings}%
\usepackage{setspace}
\usepackage{lipsum}
\usepackage{extarrows}
\usepackage[all,cmtip]{xy}
\usepackage{lastpage}
\usepackage{soul}
\usepackage{diagrams}
\diagramstyle[labelstyle=\scriptstyle]
\usepackage{spectralsequences}
\usepackage{caption}
\usepackage{subcaption}
\usepackage{mwe}

\newcommand{\cat}[1]{\mathbf{\boldsymbol{\mathsf{#1}}}}

\usetikzlibrary{positioning}
\newcommand\restr[2]{{
  \left.\kern-\nulldelimiterspace 
  #1 
  \vphantom{\big|} 
  \right|_{#2} 
  }}
\let\oldproof\proof
\renewcommand{\proof}{\color{black}\oldproof} 
\normalsize
\makeatletter

\renewcommand{\lim}{\mathsf{lim}}

\renewcommand{\hom}{\hspace{0.1em}\mathsf{Hom}}

\DeclareMathOperator{\id}{\mathsf{id}}
\DeclareMathOperator{\End}{\mathsf{End}}

\newtheorem{lemma}{Lemma}
\newtheorem*{pflemma}{\indent Proof of lemma}
\let\oldpflemma\pflemma
\renewcommand{\pflemma}{\color{black}\oldpflemma}

\makeatletter
\newcommand\dirlim{\mathop{\mathpalette\varlim@{\rightarrowfill@\scriptstyle}}\nmlimits@}
\newcommand\invlim{\mathop{\mathpalette\varlim@{\leftarrowfill@\scriptstyle}}\nmlimits@}
\makeatother

\begingroup

\diagramstyle[labelstyle=\scriptstyle]
\begingroup\lccode`~=`& \lowercase{\endgroup
\newcommand{\spmat}[1]{%
  \left[
  \let~=&
  \begin{smallmatrix}#1\end{smallmatrix}
  \right]
}}

\theoremstyle{thmstyleone}%
\newtheorem{theorem}{Theorem}
\newtheorem{proposition}[theorem]{Proposition}%

\theoremstyle{thmstyletwo}%
\newtheorem{example}{Example}%

\theoremstyle{thmstylethree}%
\newtheorem{definition}{Definition}%

\begin{document}
\title[Notes on Pointwise Finite-Dimensional $2$-Parameter Persistence Modules]{Notes on Pointwise Finite-Dimensional $2$-Parameter Persistence Modules}


\author*[1]{\fnm{Wenwen} \sur{Li}}\email{wli11uco@gmail.com}
\author[2]{\fnm{Murad} \sur{\"Ozayd\i n}}\email{mozaydin@ou.edu}

\affil*[1]{\orgdiv{Department of Mathematics and Computer Science}, \orgname{Hobart and William Smith Colleges}, \orgaddress{\street{47645 College Dr}, \city{Geneva}, \postcode{14456}, \state{New York}, \country{USA}}}
\affil[2]{\orgdiv{Department of Mathematics}, \orgname{University of Oklahoma}, \orgaddress{\street{601 Elm Ave}, \city{Norman}, \postcode{73019}, \state{Oklahoma}, \country{USA}}}



\abstract{In this paper, we study pointwise finite-dimensional (p.f.d.) $2$-parameter persistence modules where each module admits a finite convex isotopy subdivision. We show that a p.f.d. $2$-parameter persistence module $M$ (with a finite convex isotopy subdivision) is isomorphic to a $2$-parameter persistence module $N$ where the restriction of $N$ to each chamber of the parameter space $(\mathbb{R},\leq)^2$ is a constant functor. Moreover, we show that the convex isotopy subdivision of $M$ induces a finite encoding of $M$. Finally, we prove that every indecomposable thin $2$-parameter persistence module is isomorphic to a polytope module.}




\maketitle
\tableofcontents
\section{Introduction}\label{sec1}
Persistent homology is a mathematical concept widely applied to analyze the topological features of data clouds. Let $\mathbb{F}$ be a field and $(P,\leq)$ be the set $P$ with a partial order. A \textit{persistence module} over $(P,\leq)$ is a family of $\mathbb{F}$-vector spaces $\{Mi\mid i\in P\}$ and a doubly-indexed family of linear maps $\{M(i\leq j): Mi\rightarrow Mj\mid i\leq j\}$ such that $M(j\leq k)\circ M(i\leq j)=M(i\leq k)$ for any $i\leq j\leq k$ in $P$ and $M(i\leq i)=\mathrm{id}_{Mi}$ for all $i\in P$. Equivalently, $M$ can be viewed as a functor $M:(P,\leq)\rightarrow \cat{Vect}_{\mathbb{F}}$. We use $\cat{Vect}_{\mathbb{F}}^{(P,\leq)}$ to denote the category of persistence modules over $(P,\leq)$. A persistence module $M$ is called \textit{pointwise finite-dimensional} (p.f.d.) when $Mi$ is a finite-dimensional vector space for all $i\in P$, and we denote the full subcategory of $\cat{Vect}_{\mathbb{F}}$ consisting of all finite-dimensional vector spaces over $\mathbb{F}$ as $\cat{vect}_{\mathbb{F}}$. In particular, when $(P,\leq)=(\mathbb{R},\leq)^n$ with the product order\footnote{In other words, $(a_1,\dots, a_n)\leq (b_1,\dots, b_n)\in(\mathbb{R},\leq)^n$ if and only if $a_i\leq b_i$ for all $i=1,\dots,n$.}, the objects of $\cat{vect}_{\mathbb{F}}^{(\mathbb{R},\leq)^n}$ are called \textit{$n$-parameter p.f.d. persistence modules}.

Every $n$-parameter p.f.d. persistence module $M$ is equivalent to a $\mathbb{R}^n$-graded module over $\mathbb{F}[0,+\infty)^n$~\cite{lesnick2015theory}. When $n=1$, an algebraic invariant (called \textit{barcode}) of $M$ can be obtained by assigning an interval to each direct summand of $M$. Botnan and Crawley-Boevey prove that the barcode is well-defined~\cite{botnan2020decomposition}. Barcode can be viewed as the persistence analog of the Betti number~\cite{ghrist2008barcodes}.

When $n\geq 2$, there is no structure theorem available for $\mathbb{R}^n$-graded modules over $\mathbb{F}[0,+\infty)^n$. Therefore, there is no algebraic characterization of multiparameter p.f.d. persistence modules that is analogous to the barcodes of the $1$-parameter p.f.d. persistence modules. Moreover, the indecomposable submodules of a multiparameter p.f.d. persistence module can be very complicated. For example, Buchet and Escolar show that any $n$-parameter persistence module can be embedded as a slice of an indecomposable $(n+1)$-parameter persistence module~\cite{BE-2020}~\cite{BE-2019}. 

In this paper, we first present some results about the $2$-parameter p.f.d. persistence modules that admit a \textit{finite convex isotopy subdivision} (Definition \ref{def-isotopy-subdivision}). Moreover, we present a result regarding the indecomposable thin $2$-parameter persistence modules.

\subsection*{Main Results} Here are the main results of this paper: 
\begin{theorem}\label{thm-sim-1}
Given $M\in\cat{vect}_{\mathbb{F}}^{(\mathbb{R},\leq)^{2} }$. Assume there exists a finite convex isotopy subdivision of $(\mathbb{R},\leq)^{2}$ subordinate to $M$. Then $M\cong N$, where $N:(\mathbb{R},\leq)^{2}\rightarrow \cat{vect}_{\mathbb{F}}$ is defined as follows:
\begin{itemize}
    \item For every chamber $J$, $Nj=\lim_JM$ for all $j\in J$;
    \item For all $j_1\leq j_2$, let $J_1$ be the chamber containing $j_1$, and $J_2$ be the chamber containing $j_2$. Then $N(j_1\leq j_2): Nj_1\rightarrow Nj_2$ is
    $$N(j_1\leq j_2)=\begin{cases} \id_{\lim_JM}, & \mbox{if } J:=J_1=J_2 \\\phi_{J_1J_2} , & \mbox{else } \end{cases}$$
\end{itemize}
\end{theorem}

The construction of linear transformation $\phi_{J_1J_2}$ is provided in Proposition~\ref{prop-sim-2}.

Because the functor category $\cat{vect}_{\mathbb{F}}^{(\mathbb{R},\leq)^{2}}$ is a Krull-Schmidt category~\cite{botnan2020decomposition}, it follows that to determine the indecomposable direct summands of $M$, it suffices to find the indecomposable direct summands of $N$.

\begin{theorem}\label{thm-sim-1'}
Given $M\in\cat{vect}_{\mathbb{F}}^{(\mathbb{R},\leq)^{2} }$. Assume there exists a finite convex isotopy subdivision of $(\mathbb{R},\leq)^{2}$ subordinate to $M$, and let $(P,\leq)$ be the poset of chambers. Then $M\cong \widetilde{N}\circ \mathcal{F}_M$, where $\widetilde{N}:(P,\leq)\rightarrow \cat{vect}_{\mathbb{F}}$ is defined as follows:
\begin{itemize}
    \item Given $p\in P$, $\widetilde{N}_p=\lim_{J_p}M\cong Mi$ for all $i\in J_p$;
    \item For all $p\leq q\in P$, $$\widetilde{N}(p\leq q)=\begin{cases} \id_{\lim_{J_p}M}, & \mbox{if } p=q \\ \phi_{J_pJ_q} , & \mbox{else } \end{cases}$$
where $\phi_{J_pJ_q}$ is the canonical linear transformation defined in Proposition \ref{prop-sim-2}.
\end{itemize} 
\end{theorem}

Therefore, any finite convex isotopy subdivision of $(\mathbb{R},\leq)^{2}$ subordinate to $M$ is also a constant subdivision of $(\mathbb{R},\leq)^{2}$ subordinate to $M$ and a finite encoding of $M$.

Next, we study thin $2$-parameter persistence modules. It is well-known that interval modules, as objects of $\cat{vect}_{\mathbb{F}}^{(P,\leq)}$, are indecomposable. In particular, when $(P,\leq)=(\mathbb{R},\leq)^2$, polytope modules are indecomposable. However, the converse of this statement does not hold in general. By imposing additional assumptions, the converse may be true up to isomorphism. For example, it is known that every thin indecomposable persistence module on a finite 2-dimensional grid is isomorphic to an interval module~\cite{asashiba2018interval}. In this paper, we extend this result to indecomposable thin persistence modules over $(\mathbb{R},\leq)^2$:

\begin{theorem}\label{chap3.3:thm}
Every indecomposable thin persistence module $M\in\cat{vect}_{\mathbb{F}}^{(\mathbb{R}, \leq)^2}$ is isomorphic to a polytope module.
\end{theorem}

Theorem \ref{thm-sim-1'} and Theorem \ref{chap3.3:thm} may fail when $M$ has more than two parameters. See Example \ref{ex-dim3}. 
\section{Preliminaries: Persistence Modules Over Posets}\label{prelim}



Let $(P,\leq)$ be a poset. A subset $\cat{I}$ of $P$ is called \textit{convex} if $i,k\in\cat{I}$ implies $j\in\cat{I}$ for any $i\leq j\leq k$ in $P$, and is called \textit{connected} if for all $a,b\in\cat{I}$, there exists $x_0:=a, x_1,\dots, x_{n-1}, x_n:=b$ in P such that $x_i$ and $x_{i+1}$ are comparable for all $i=0,\dots, n-1$. The set $\cat{I}$ is called an \textit{interval} of $(P,\leq)$ if $\cat{I}$ is connected and convex.

Assume $(P,\leq)=(\mathbb{R},\leq)^2$ and let $p: a- x_1- x_2-\cdots - x_n- b$ be a zigzag path in $J\subseteq \mathbb{R}^2$, where $-$ is $\leq$ or $\geq$. $p$ is called a \textit{staircase} is $p$ consists only of horizontal and vertical arrows in $(\mathbb{R},\leq)^{2}$. Moreover, $p$ is called a \textit{reduced staircase} if it contains no two consecutive horizontal or vertical arrows. When $J \subseteq \mathbb{R}^2$ is convex (as a poset), there exists a reduced staircase for every zigzag path $p$ in $J$. Since $(\mathbb{R},\leq)^2$ is a thin category, we may assume $p$ is a reduced staircase.

Let $(P,\leq)$ be a poset and $\cat{I}\subseteq (P,\leq)$ is an interval. Define the \textit{interval module} $\mathbb{F}\cat{I}$ as 
$$\mathbb{F}\cat{I}i = \begin{cases} \mathbb{F}, & \mbox{if } i\in \cat{I};\\ 0, & \mbox{if } i\notin \cat{I}\end{cases} \mbox{\quad and \quad} \hom(\mathbb{F}\cat{I}i, \mathbb{F}\cat{I}j) = \begin{cases} \{\id_{\mathbb{F}}\}, & \mbox{if } i\leq j\in \cat{I};\\ 0, & \mbox{else} \end{cases} $$
When $(P,\leq)=(\mathbb{R},\leq)^n$ and $n\geq 2$, we call $\mathbb{F}\cat{I}$ the \textit{polytope module}\index{polytope module} over $\cat{I}$.

Let $M\in \cat{vect}_{\mathbb{F}}^{(P,\leq)}$. $M$ is \textit{decomposable}\index{decomposable} if there exists non-trivial subrepresentations $N$ and $N'$ such that $Mi\cong Ni\oplus N'i$ for all $i\in P$. We say $M$ is \textit{indecomposable}\index{indecomposable} if it is not decomposable. $M$ is a \textit{thin representation}\index{thin representation} if $\dim(Mi)=0 \mbox{ or }1$ for all $i\in P$.

The following lemma about interval modules is well-known.

\begin{lemma}\label{thinpoly}
Interval modules are thin and indecomposable.
\end{lemma}
\begin{proof}
Let $M$ be an interval module. The support of $M$ is an interval of $(P,\leq)$. Consider the endomorphism ring of $M$. We want to show $\End(M)\cong\mathbb{F}$. Let $f\in\End(M)$. Note that for every $p\in (P,\leq)$, $f_p: \mathbb{F}\rightarrow \mathbb{F}$ is a linear transformation, hence $f_p(x)=c_f\cdot x$ for some $c_f\in\mathbb{F}$. Let $p'$ be a point in the support of $M$. Note that there exists a zigzag path from $p$ to $p'$ because the support of $M$ is an interval. Since $f$ is a morphism between two representations and $M$ is a polytope module, we must have $f_q(x)=c_f\cdot x$. Define $\Phi:\End(M)\rightarrow\mathbb{F}$ by $\Phi(f)=c_f$ for each $f\in \End(M)$. It is clear that $\Phi$ is bijective, so we only need to show that $\Phi$ is a ring homomorphism. Note that $\Phi(f+g)=c_f+c_g=\Phi(f)+\Phi(g)$ and $\Phi(f\circ g)=c_f\cdot g=\Phi(f)\Phi(g)$, hence $\Phi$ is a ring homomorphism. 
\end{proof}

Let $M\in\cat{vect}_{\mathbb{F}}^{(\mathbb{R},\leq)^{2}}$. A \textit{constant subdivision of $(\mathbb{R},\leq)^{2}$ subordinate to $M$}~\cite{miller2020homological} is a partition of $(\mathbb{R},\leq)^{2}$ into constant regions such that for each constant region $I$ there is a single vector space $M_I$ with $M_I\rightarrow M_i$ for all $i\in I$ that has no monodromy: if J is another constant region, then all comparable pairs $i\leq j$ with $i\in I$ and $j\in J$ induce the same composite homomorphism $M_I\rightarrow M_J$.

Fix a poset $(Q,\leq)$. An \textit{encoding}~\cite{miller2020homological} of a $Q$-module $M$ by a poset $(P,\leq)$ is a poset morphism $\pi: (Q,\leq)\rightarrow (P,\leq)$ together with a $P$-module $H$ such that $M\cong \pi^\ast H$, the pullback of $H$ along $\pi$, which is naturally a $Q$-module. The encoding is \textit{finite} if $(P,\leq)$ is finite (here $\dim(H_p)$ is also finite for all $p\in (P,\leq)$ since $M$ is p.f.d.).

\section{$2$-Parameter p.f.d. Persistence Modules with Convex Isotopy Subdivisions}
\begin{definition}[Isotopy Subdivision]\label{def-isotopy-subdivision}
    Let $M\in \cat{vect}_{\mathbb{F}}^{(\mathbb{R},\leq)^{2}}$. An isotopy subdivision of $(\mathbb{R},\leq)^{2}$ subordinate to $M$ is a partition of $(\mathbb{R},\leq)^{2}$ into connected subsets (called chambers) such that $M(a\leq b)$ is an isomorphism provided that $a\leq b$ are contained in the same chamber. An isotopy subdivision is finite if the number of chambers in the partition is finite. Moreover, we call the partition a convex isotopy subdivision if every chamber in the isotopy subdivision of $(\mathbb{R},\leq)^{2}$ is convex.
\end{definition}

Given $M\in \cat{vect}_{\mathbb{F}}^{(\mathbb{R},\leq)^{2}}$. Assume there exists a finite convex isotopy subdivision of $(\mathbb{R},\leq)^{2}$ subordinate to $M$. In this section, we show that there exists $N\in \cat{vect}_{\mathbb{F}}^{(\mathbb{R},\leq)^{2}}$ where $N(i\leq j)=\id$ for all $i$ and $j$ in a chamber such that $M\cong N$. In particular, Proposition~\ref{prop-sim-2} shows that every finite convex isotopy subdivision of $M\in \cat{vect}_{\mathbb{F}}^{(\mathbb{R},\leq)^{2}}$ is a finite constant subdivision of $(\mathbb{R},\leq)^{2}$ subordinate to $M$.

\begin{definition}\label{def-path1}
    Let $J\subseteq (\mathbb{R},\leq)^{2}$ be a connected poset and $M\in \cat{vect}_{\mathbb{F}}^J$ such that $M(i\leq j)$ is an isomorphism for all $i\leq j\in J$. For each comparable pair $i$ and $j$ in J, define $$\hat{M}(i,j)=\begin{cases} M(i\leq j), & \mbox{if } i\leq j \\ M(j\leq i)^{-1}, & \mbox{if } j\leq i \end{cases}$$ 
    An arrow $i-j$ is called a good arrow $(G)$ if $\hat{M}(i,j)=M(i\leq j)$ and is called a bad arrow $(B)$ if $\hat{M}(i,j)=M(j\leq i)^{-1}$, where $-$ means $\leq$ or $\geq$. Given $a, b\in J$, let $p: a- x_1- x_2-\cdots - x_n- b$ be a zigzag path in $J$, define $$M(a{\overset{p}{\rightsquigarrow}}b)=\hat{M}(x_n,b)\circ \hat{M}(x_{n-1},x_n)\circ\cdots \hat{M}(x_1,x_2)\circ  \hat{M}(a,x_1)$$
\end{definition}

\begin{lemma}\label{prop-lemma-0}
    Let $J\subseteq (\mathbb{R},\leq)^{2}$ be a connected and convex poset and $M\in \cat{vect}_{\mathbb{F}}^J$ such that $M(i\leq j)$ is an isomorphism for all $i\leq j\in J$. Then for any $a\leq b\in J$ and any zigzag path $p:a-x_1-\cdots-x_n-b$, $M(a{\overset{p}{\rightsquigarrow}}b)=M(a\leq b)$.
\end{lemma}
\begin{proof}
  Since $J$ is convex, we may assume $p: a- x_1- x_2-\cdots - x_n- b$ is a reduced staircase. Induction on the length of $p$ (denoted by $l(p)$). When $l(p)=1$, $M(a{\overset{p}{\rightsquigarrow}}b)=M(a\leq b)$ because $J$ is thin.

  When $l(p)=2$, there exists $x_1\in J$ such that $p:a- x_1- b$. If $a\leq x_1\leq b$, then $M(a{\overset{p}{\rightsquigarrow}}b)=M(a\leq b)$ because $J$ is thin. If $a\geq x_1\geq b$, then $a=x_1=b$ and $M(a{\overset{p}{\rightsquigarrow}}b)=\id_{Ma}=M(a\leq b)$. If $a\leq x_1\geq b$, then $M(a\leq x_1)=M(b\leq x_1)\circ M(a\leq b)$. Hence $M(a\leq b)=M(b\leq x_1)^{-1}\circ M(a\leq x_1)=M(a{\overset{p}{\rightsquigarrow}}b)$. Similarly, if $a\geq x_1\leq b$, then $M(x_1\leq b)=M(a\leq b)\circ M(x_1\leq a)$. Hence $M(a\leq b)=M(x_1\leq b)\circ M(x_1\leq a)^{-1}=M(a{\overset{p}{\rightsquigarrow}}b)$.

Assume the statement is true for all zigzag paths of length less or equal to $N$. Let $p: a- x_1- x_2-\cdots - x_{N}- b$ be a zigzag path from $a$ to $b$ of length $N+1$. If $p$ has a self-intersection $c$, add a new vertex $c$ to $p$. Then we obtain a new path $\tilde{p}$. There are 4 possible forms of $\tilde{p}$:
\begin{description}
       \item[(1) ] $a- x_1-\cdots - x_i- c-x_{i+1}- -\cdots - x_j-c-x_{j+1}-\cdots - x_{N}- b$ for some $i<j$;
    \item[(2) ] $a-c- x_1-\cdots - x_j-c-x_{j+1}-\cdots - x_{N}- b$ for some $j$;
    \item[(3) ] $a- x_1-\cdots - x_i- c-x_{i+1}-\cdots - x_{N}-c- b$ for some $i$;
    \item[(4) ] $a-c- x_1-\cdots - x_{N}-c- b$. 
\end{description}

When $p$ is of the form 1, 2, or 3, then the induction hypothesis implies that $M(c-\dots-c)=\id_{Mc}$ and hence $M(a\leq b)=M(a{\overset{p}{\rightsquigarrow}}b)$. When  $p$ is of the form 4, without loss of generality, we may assume $c\leq x_1$. The induction hypothesis implies that $M(c\leq x_1)=M(c-x_N-x_{N-1}-\cdots-x_1)$. Hence $M(c-x_1-\cdots-x_N-c)=\id_{Mc}$. Note that $a-c-b$ is a path in $J$ of length $2$, the induction hypothesis implies that $\hat{M}(c,b)\circ \hat{M}(a,c)=M(a\leq b)$. Therefore, 
\begin{equation}
    \begin{aligned}
       M(a{\overset{p}{\rightsquigarrow}}b)&=\hat{M}(c,b)\circ \hat{M}(x_N,c)\circ \hat{M}(x_{N-1},x_N)\circ \cdots \hat{M}(c,x_1)\circ \hat{M}(a,c)\\
       &=\hat{M}(c,b)\circ \id_{Mc}\circ \hat{M}(a,c)\\
       &=M(a\leq b)
    \end{aligned}
\end{equation}

Now we assume $p: a- x_1- x_2-\cdots - x_{N}- b$ does not have self-intersection. Note that $a$ and $b$ are not comparable if $p$ consists of alternating good (G) and bad (B) arrows. Hence, every arrow of $p$ is good, or $p$ contains a subpath of the form: BBG, GBB, BGG, or GGB. When arrows of $p$ are all good arrows, then $M(a{\overset{p}{\rightsquigarrow}}b)=M(a\leq b)$ because $J$ is thin. Now we assume $p$ contains a subpath $x_{i-1}-x_i-x_{i+1}-x_{i+2}$ of the form BBG. Because $J\subseteq (\mathbb{R},\leq)^{2}$, there exists a bad arrow $c-d$ where (1) $c=x_{i-1}$ and $x_{i+1}< d \leq x_{i+2}$, or (2) $x_i< c \leq x_{i-1}$ and $d=x_{i+2}$. Substitute the subpath $x_{i-1}-x_i-x_{i+1}-x_{i+2}$ by (1) $c-d-x_{i+2}$ or (2) $x_{i-1}-c-d$, we obtain a new path $q$ of length at most $N$. The induction hypothesis implies that $M(a{\overset{q}{\rightsquigarrow}}b)=M(a\leq b)$ and (1) $M(x_{i-1}-x_i-x_{i+1}-x_{i+2})=M(c-d-x_{i+2})$ or (2) $M(x_{i-1}-x_i-x_{i+1}-x_{i+2})=M(x_{i-1}-c-d)$. Therefore, $M(a{\overset{p}{\rightsquigarrow}}b)=M(a\leq b)$. A similar argument shows that the statement is true when $p$ contains a subpath of the form GBB, BGG, or GBB.

By induction, we conclude that $M(a{\overset{p}{\rightsquigarrow}}b)=M(a\leq b)$ for all zigzag path $p$ in $J$.
\end{proof} 

\begin{lemma}\label{prop-lemma-1}
    Let $J\subseteq (\mathbb{R},\leq)^{2}$ be a connected and convex poset and $M\in \cat{vect}_{\mathbb{F}}^J$ such that $M(i\leq j)$ is an isomorphism for all $i\leq j\in J$. Then for any $a,b\in J$, $M(a{\overset{p}{\rightsquigarrow}}b)$ does not depend on the choice of zigzag path $p$ connecting $a$ and $b$ in $J$.
\end{lemma}
  \begin{proof}
  For all $a,b\in J$, let $p: a- x_1- x_2-\cdots - x_n- b$ and $q:a- y_1- y_2-\cdots - y_m- b$ be two zigzag paths from $a$ to $b$. The concatenation of $p$ and the inverse\footnote{The inverse of a zigzag path $q$, denoted by $q^{-1}$, is a zigzag path obtained by reversing the direction of each edge in $q$.} of $q$, denoted by $p\ast q^{-1}$, is the zigzag path $$p\ast q^{-1}: a- x_1- x_2-\cdots - x_n- b -y_m -\cdots -y_2 - y_1-a$$
  Moreover, $M(a{\overset{\hspace{5pt}p\ast q^{-1}}{\rightsquigarrow}}a) =M(a{\overset{q}{\rightsquigarrow}}b)^{-1}\circ M(a{\overset{p}{\rightsquigarrow}}b)$ because $$M(a{\overset{q}{\rightsquigarrow}}b)^{-1}=\hat{M}(a,y_1)^{-1}\circ \hat{M}(y_1,y_2)^{-1}\circ\cdots\circ \hat{M}(y_m,b)^{-1}$$
  Lemma~\ref{prop-lemma-0} implies that $M(a{\overset{q}{\rightsquigarrow}}b)^{-1}\circ M(a{\overset{p}{\rightsquigarrow}}b)=\id_{Ma}$. Therefore, $M(a{\overset{p}{\rightsquigarrow}}b)=M(a{\overset{q}{\rightsquigarrow}}b)$.
  \end{proof}

Consequently, when the assumptions in Lemma~\ref{prop-lemma-1} are satisfied, we drop $p$ from the notation $M(a{\overset{p}{\rightsquigarrow}}b)$ in Definition~\ref{def-path1}.

\begin{proposition}\label{prop-sim-1}
    Let $J\subseteq (\mathbb{R},\leq)^{2}$ be a connected and convex poset and $M\in \cat{vect}_{\mathbb{F}}^J$ such that $M(i\leq j)$ is an isomorphism for all $i\leq j\in J$. Then $\lim_{J}M\cong Ma$ for all $a\in J$.
\end{proposition}
\begin{proof}
    Given $a\in J$, let $Ma: J\rightarrow \cat{vect}_{\mathbb{F}}$ be a constant functor where $Ma(i)=Ma$ for all $i\in J$ and $Ma(i\leq j)=\id_{Ma}$ for all $i\leq j\in J$. Define $\sigma: Ma\Rightarrow M$, where $\sigma_i:=M(a\rightsquigarrow i)$ for all $i\in J$. (In particular, $\sigma_a=\id_{Ma}$.) Lemma~\ref{prop-lemma-1} implies that $\sigma_i$ is well-defined. For any $i\leq j\in J$, by Lemma~\ref{prop-lemma-1}, $M(i\leq j) \circ M(a\rightsquigarrow i)=M(a\rightsquigarrow j)$. Hence $\sigma_j=M(i\leq j)\circ\sigma_i$ for all $i\leq j\in J$. Therefore, $\sigma:Ma\Rightarrow M$ is a cone.

    Assume there is a cone $\tau: N\Rightarrow M$. Define $\phi: N\rightarrow Ma$ by $\phi:=\tau_a$. Given $i\in J$, let $a- x_1 - \cdots - x_n - i$ be a zigzag path in $J$. Note that 
    \begin{equation}
        \begin{aligned}
          \sigma_i\circ \phi&=M(a\rightsquigarrow i)\circ\tau_a\\&=\hat{M}(x_n, i)\circ \hat{M}(x_{n-1}, x_n)\circ\cdots\circ\hat{M}(x_1, x_2)\circ\hat{M}(a, x_1)\circ\tau_a\\
          &=\hat{M}(x_n, i)\circ \hat{M}(x_{n-1}, x_n)\circ\cdots\circ\hat{M}(x_1, x_2)\circ\tau_{x_1}\\
          &\quad\vdots\\
          &=\tau_i
        \end{aligned}
    \end{equation}
    
    Therefore, $\phi$ is a morphism from the cone $\tau$ to the cone $\sigma$.
    
    If there exists another morphism $\psi: N\rightarrow Ma$ such that $\tau_i=\sigma_i\circ\psi$ for all $i\in J$, then $\phi:\tau_a=\sigma_a\circ\psi=\id_{Ma}\circ\psi=\psi$. 

    In conclusion, $\sigma$ is a limit cone and $\lim_{J}M\cong Ma$ for all $a\in J$.
\end{proof}

Next, we construct the morphisms between limits.

\begin{lemma}\label{lemma:no-inverse}
Let $a, a'\in J_1$ and $b, b'\in J_2$ such that $a\leq b$ and $a'\geq b'$, where $J_1$ and $J_2$ are convex and connected subsets of $(\mathbb{R},\leq)^{2}$. Then $J_1\cap J_2\neq\emptyset$.
\end{lemma}
\begin{proof}
Since $J_1$ and $J_2$ are convex, the statement is true when $a=a'$ or $b=b'$. Now we assume $a\neq a'$ and  $b\neq b'$. Let $p: a-x_1-\cdots- x_n-a'$ be a zigzag path in $J_1$ and $q: b-y_1-\cdots- y_m-b'$ be a zigzag path in $J_2$. Since $J_1$ and $J_2$ are convex, we may assume $p$ and $q$ are staircases with no two consecutive horizontal or vertical arrows. Induction on $l(p)+l(q)$.

When $l(p)+l(q)=2$, there are 4 cases: (1) $a\leq a'$ and $b\leq b'$; (2) $a\geq a'$ and $b\leq b'$; (3) $a\leq a'$ and $b\geq b'$; (4) $a\geq a'$ and $b\geq b'$. Since $J_1$ and $J_2$ are convex, the statement is true for cases (1), (2), and (4). For case (3), note that $a,a',b,b'\in \mathbb{R}^2$, there exists $c\in \mathbb{R}^2$ such that $a\leq c\leq a'$ and $b'\leq c\leq b$. The convexity of  $J_1$ and $J_2$ implies that $c\in J_1\cap J_2$.


Assume the statement is true when the sum of the lengths of two zigzag paths $a{\overset{}{\rightsquigarrow}}a'$ and $b{\overset{}{\rightsquigarrow}}b'$ is less than or equal to $N$. When $l(p)+l(q)\leq N+1$, there are $3$ cases: (1) the arrows in $p$ are all good (or all bad, respectively), or the arrows in $q$ are all good (or all bad, respectively); (2) $p$ or $q$ contains a subpath of the form BBG, GBB, BGG, or GGB; (3) $p$ and $q$ consist of alternating arrows. For case (1), without loss of generality, we may assume all arrows of $p$ are good. We may replace $p$ with path $p': a\rightarrow a'$. Note that $l(p')+l(q)\leq N$, by induction hypothesis, $J_1\cap J_2\neq \emptyset$. Case (2) can be proved using similar arguments as in Lemma~\ref{prop-lemma-0} because we can replace those subpaths with a new subpath of length at most $2$. For case (3), note that $a\leq b$ and $a'\geq b'$, $p$ and $q$ intersects in $\mathbb{R}^2$. Therefore, the intersection of $p$ and $q$ is in $J_1$ and $J_2$, which implies $J_1\cap J_2\neq \emptyset$.
\end{proof}

\begin{definition}\label{def-path2}
 Let $J_1, J_2\subseteq (\mathbb{R},\leq)^{2}$ be connected and convex posets such that $J_1\cap J_2=\emptyset$. Given $M\in \cat{vect}_{\mathbb{F}}^J$ such that $M(i\leq j)$ is an isomorphism for all $i\leq j\in J_k, k=1,2$, assume there exists $a_0\in J_1$ and $b_0\in J_2$ such that $a_0\leq b_0$. Then for any $a\in J_1$ and $b\in J_2$, define $$M(a{\overset{}{\rightsquigarrow}}b)=M(b{\overset{}{\rightsquigarrow}}b_0)^{-1} \circ M(a_0\leq b_0) \circ M(a{\overset{}{\rightsquigarrow}}a_0)$$ 
 where $a{\overset{}{\rightsquigarrow}}a_0$ is a zigzag path in $J_1$ and $b{\overset{}{\rightsquigarrow}}b_0$ is a zigzag path in $J_2$.  
\end{definition}

\begin{proposition}\label{prop-sim-1.5}
$M(a{\overset{}{\rightsquigarrow}}b)$ is well-defined. That is, $M(a{\overset{}{\rightsquigarrow}}b)$ does not depend on the choice of $a{\overset{}{\rightsquigarrow}}a_0$ and $b{\overset{}{\rightsquigarrow}}b_0$. Moreover, if there exists $a_0'\in J_1$ and $b_0'\in J_2$ such that $a_0'\leq b_0'$, then $M(b{\overset{}{\rightsquigarrow}}b_0)^{-1} \circ M(a_0\leq b_0) \circ M(a{\overset{}{\rightsquigarrow}}a_0)=M(b{\overset{}{\rightsquigarrow}}b'_0)^{-1} \circ M(a'_0\leq b'_0) \circ M(a{\overset{}{\rightsquigarrow}}a'_0)$.
\end{proposition}

We need the following lemmas to prove Proposition~\ref{prop-sim-1.5}.

\begin{lemma}\label{prop-lemma-2}
    Let $J\subseteq (\mathbb{R},\leq)^{2}$ be a connected and convex poset and $M\in \cat{vect}_{\mathbb{F}}^J$ such that $M(i\leq j)$ is an isomorphism for all $i\leq j\in J$. Then for any $a \in (\mathbb{R},\leq)^{2}$ and $b,c\in J$ such that $a\leq b$ and $a\leq c$, then $M(b{\overset{}{\rightsquigarrow}}c)\circ M(a\leq b)=M(a\leq c)$, where $b{\overset{}{\rightsquigarrow}}c$ is any zigzag path in $J$ from $b$ to $c$.

    Dually, for any $a \in (\mathbb{R},\leq)^{2}$ and $b,c\in J$ such that $a\geq b$ and $a\geq c$, then $M(c\leq a) \circ  M(b{\overset{}{\rightsquigarrow}}c)=M(b\leq a)$, where $b{\overset{}{\rightsquigarrow}}c$ is any zigzag path in $J$ from $b$ to $c$.
\end{lemma}  

\begin{proof}
Let $p:b-x_1-\cdots-x_n-c$ be a zigzag path in $J$. Since $J$ is convex, we may assume $p$ is a reduced staircase. Induction on the length of $p$ (denoted by $l(p)$). When $l(p)=1$, $b$ and $c$ are comparable. Without loss of generality, we assume $b\leq c$. Note that $J$ is thin, hence $M(b\leq c)\circ M(a\leq b)=M(a\leq c)$.

When $l(p)=2$, there are 4 cases: (1) $b\leq x_1\leq c$; (2) $b\leq x_1\geq c$; (3) $b\geq x_1\leq c$; (4) $b\geq x_1\geq c$. Note that $p$ is a staircase in $J\subseteq (\mathbb{R},\leq)^{2}$ where both endpoints of $p$ is no smaller than $a$, hence $a\leq x_1$. Since $J$ is thin, we obtain $\hat{M}(b,x_1)\circ M(a\leq b)=M(a\leq x_1)$ and $\hat{M}(x_1,c)\circ M(a\leq x_1)=M(a\leq c)$, which implies $M(b{\overset{p}{\rightsquigarrow}}c)\circ M(a\leq b)=M(a\leq c)$.

Now we assume the statement is true when the length of any zigzag path (connecting $b$ and $c$) is no more than $N$. Let $p:b-x_1-x_2-\cdots-x_N-c$ be a zigzag path from $b$ to $c$ of length $N+1$. There are 3 cases: (1) $p$ has a self-intersection; (2) $p$ has no self-intersection but contains one of the following subpaths: BBG, GBB, BGG, or GGB; (3) $p$ has no self-intersection, and the arrows of $p$ alternate.

Cases (1) and (2) can be proved using similar arguments as in Lemma~\ref{prop-lemma-0} because we can replace those subpaths with a new subpath of length at most $2$. For case (3), note that $a\leq b$ and $a\leq c$ in $(\mathbb{R},\leq)^{2}$, there exists $i\in\{1,\dots, N\}$ such that $x_i$ is comparable to $a$. The induction hypothesis implies that $M(b{\overset{p}{\rightsquigarrow}}x_i)\circ M(a\leq b)=M(a\leq x_i)$ and $M(x_i{\overset{p}{\rightsquigarrow}}c)\circ M(a\leq x_i)=M(a\leq c)$. Therefore, $M(b{\overset{p}{\rightsquigarrow}}c)\circ M(a\leq b)=M(a\leq c)$. 

By induction, we conclude that $M(b{\overset{p}{\rightsquigarrow}}c)\circ M(a\leq b)=M(a\leq c)$ for all zigzag $p$ in $J$.
\end{proof}

\begin{lemma}
Let $a, a'\in J_1$ and $b, b'\in J_2$ such that $a\leq b$ and $a'\leq b'$, where $J_1$ and $J_2$ are as defined above. Then $M(b{\overset{}{\rightsquigarrow}}b')\circ M(a\leq b)=M(a'\leq b')\circ M(a{\overset{}{\rightsquigarrow}}a')$.
\end{lemma}

\begin{proof}
Let $p: a-x_1-\cdots- x_n-a'$ be a zigzag path in $J_1$ and $q: b-y_1-\cdots- y_m-b'$ be a zigzag path in $J_2$. Since $J_1$ is convex, we may assume $p$ and $q$ are reduced staircases. Induction on the length of $p$. When $l(p)=0$, Lemma~\ref{prop-lemma-2} implies that the statement is true. When $l(p)=1$, $a$ and $a'$ are comparable. Without loss of generality, we assume $a\leq a'$. Note that $a'\leq b'$, Lemma~\ref{prop-lemma-2} implies that $M(a\leq b')=M(b{\overset{}{\rightsquigarrow}}b')\circ M(a\leq b)$ and $M(a\leq b')=M(a'\leq b')\circ M(a{\overset{}{\rightsquigarrow}}a')$. Therefore, the statement is true $l(p)=1$.

When $l(p)=2$, there are 4 cases: (1) $a\leq x_1\leq a'$; (2) $a\leq x_1\geq a'$; (3) $a\geq x_1\leq a'$; (4) $a\geq x_1\geq a'$. Note that $J_1$ is thin, Lemma~\ref{prop-lemma-2} implies that the statement is true for cases (1), (3), and (4). For case (2), there are two subcases: (i) $x_1$ is comparable to $b$ or $b'$ (ii) $x_1$ is not comparable to $b$ and $b'$. For case (i), Since $J_1\cap J_2=\emptyset$, Lemma~\ref{lemma:no-inverse} implies $x_1\leq b$ or $x_1\leq b'$. Without loss of generality, we assume $x_1\leq b$. Note that $(\mathbb{R},\leq)^{2}$ is thin, hence $M(a\leq b)=M(x_1\leq b)\circ M(a\leq x_1)$. Note that $a'\rightarrow x_1$ is a path of length $1$, therefore, $M(b{\overset{}{\rightsquigarrow}}b')\circ M(x_1\leq b)=M(a'\leq b')\circ M(a'\leq x_1)=M(a'\leq b')\circ M(x_1{\overset{}{\rightsquigarrow}}a')$. For case (ii), note that there is no zigzag path in $(\mathbb{R},\leq)^{2}-\uparrow x_1-\downarrow x_1$ connecting $b$ and $b'$, where $\uparrow x_1$ is the principal upset at $x_1$ and $\downarrow x_1$ is the principal downset at $x_1$. Therefore, there exists $y\in\{b,b',y_1,\dots,y_m\}$ such that $x_1\leq y$ or $x_1\geq y$. Since $J_1\cap J_2=\emptyset$, Lemma~\ref{lemma:no-inverse} implies $x_1\leq y$. Since the length of path $a-x_1$ and $x_1-a'$ is $1$, we obtain $M(x_1\leq y)\circ M(a\leq x_1)=M(b{\overset{}{\rightsquigarrow}}y)\circ M(a\leq b)$ 
and 
$M(y{\overset{}{\rightsquigarrow}}b')\circ M(x_1\leq y)= M(a'\leq b')\circ M(a'\leq x_1)^{-1}$. Therefore
\begin{equation}
    \begin{aligned}
        M(a'\leq b')\circ M(a'\leq x_1)^{-1}\circ M(a\leq x_1)&=M(y{\overset{}{\rightsquigarrow}}b')\circ M(x_1\leq y)\circ M(a\leq x_1)\\
        &=M(y{\overset{}{\rightsquigarrow}}b')\circ M(b{\overset{}{\rightsquigarrow}}y)\circ M(a\leq b)\\
        &=M(b{\overset{}{\rightsquigarrow}}b')\circ M(a\leq b)
    \end{aligned}
\end{equation}

Now we assume the statement is true when the length of the zigzag path is less or equal to $N$. Let $p: a-x_1-\cdots- x_N-a'$ be a zigzag path in $J_1$. There are 3 cases: (1) $p$ has a self-intersection; (2) $p$ has no self-intersection but contains one of the following subpaths: BBG, GBB, BGG, or GGB; (3) $p$ has no self-intersection, and the arrows of $p$ alternate.

Cases (1) and (2) can be proved using similar arguments as in Lemma~\ref{prop-lemma-0}. For case (3), there are two subcases: (i) $x_i$ is comparable to $b$ or $b'$ for some $i\in\{1,\dots, N\}$ (ii) $x_i$ is not comparable to $b$ and $b'$ for all $i\in\{1,\dots, N\}$.  For case (i), Since $J_1\cap J_2=\emptyset$, Lemma~\ref{lemma:no-inverse} implies $x_i\leq b$ or $x_i\leq b'$. Without loss of generality, we assume $x_i\leq b$. Lemma~\ref{prop-lemma-2} implies $M(a\leq b)=M(x_1\leq b')\circ M(a{\overset{}{\rightsquigarrow}}x_1)$, and the induction hypothesis implies $M(b{\overset{}{\rightsquigarrow}}b')\circ M(x_1\leq b)=M(a'\leq b')\circ M(x_1{\overset{}{\rightsquigarrow}}a')$. Therefore,
\begin{equation}
    \begin{aligned}
        M(b{\overset{}{\rightsquigarrow}}b')\circ M(a\leq b)
        &=M(b{\overset{}{\rightsquigarrow}}b')\circ M(x_1\leq b')\circ M(a{\overset{}{\rightsquigarrow}}x_1)\\
        &=M(a'\leq b')\circ M(x_1{\overset{}{\rightsquigarrow}}a')\circ M(a{\overset{}{\rightsquigarrow}}x_1)\\
        &=M(a'\leq b')\circ M(a{\overset{}{\rightsquigarrow}}a')
    \end{aligned}
\end{equation}

For case (ii), note that there is no zigzag path connecting $b$ and $b'$ in $(\mathbb{R},\leq)^{2}-\uparrow x_i-\downarrow x_i$ for any $i\in\{1,\dots, N\}$, there exists $y\in \{b,b',y_1,\dots, y_m\}$ such that $x_i\leq y$ or $x_i\geq y$. Since $J_1\cap J_2=\emptyset$, Lemma~\ref{lemma:no-inverse} implies $x_i\leq y$. The induction hypothesis implies that $M(b{\overset{}{\rightsquigarrow}}y)\circ M(a\leq b)= M(x_i\leq y)\circ M(a{\overset{}{\rightsquigarrow}}x_i)$ 
and
$M(y{\overset{}{\rightsquigarrow}}b')\circ M(x_i\leq y)=M(a'\leq b')\circ M(x_i{\overset{}{\rightsquigarrow}}a')$. Therefore, 
\begin{equation}
    \begin{aligned}
        M(a'\leq b')\circ M(a{\overset{}{\rightsquigarrow}}a')&=M(a'\leq b')\circ M(x_i{\overset{}{\rightsquigarrow}}a')\circ M(a{\overset{}{\rightsquigarrow}}x_i)\\
        &=M(y{\overset{}{\rightsquigarrow}}b')\circ M(x_i\leq y)\circ M(a{\overset{}{\rightsquigarrow}}x_i)\\
        &=M(y{\overset{}{\rightsquigarrow}}b')\circ M(b{\overset{}{\rightsquigarrow}}y)\circ M(a\leq b)\\
        &=M(b{\overset{}{\rightsquigarrow}}b')\circ M(a\leq b)
    \end{aligned}
\end{equation}

By induction, we conclude that $M(b{\overset{}{\rightsquigarrow}}b')\circ M(a\leq b)=M(a'\leq b')\circ M(a{\overset{}{\rightsquigarrow}}a')$.
\end{proof}

\begin{proof}[Proof of Proposition~\ref{prop-sim-1.5}]
    Note that Lemma~\ref{prop-lemma-1} implies that $M(a_0{\overset{}{\rightsquigarrow}}a'_0)\circ M(a{\overset{}{\rightsquigarrow}}a_0)=M(a{\overset{}{\rightsquigarrow}}a'_0)$ and $M(b_0{\overset{}{\rightsquigarrow}}b'_0)\circ M(b{\overset{}{\rightsquigarrow}}b_0)=M(b{\overset{}{\rightsquigarrow}}b'_0)$.
    Hence   
    \begin{equation}
        \begin{aligned}
            M&(b{\overset{}{\rightsquigarrow}}b'_0)^{-1} \circ M(a'_0\leq b'_0) \circ M(a{\overset{}{\rightsquigarrow}}a'_0)\\
            &= M(b{\overset{}{\rightsquigarrow}}b_0)^{-1}\circ M(b_0{\overset{}{\rightsquigarrow}}b'_0)^{-1} \circ M(a'_0\leq b'_0) \circ M(a_0{\overset{}{\rightsquigarrow}}a'_0)\circ M(a{\overset{}{\rightsquigarrow}}a_0)\\
            &= M(b{\overset{}{\rightsquigarrow}}b_0)^{-1}\circ M(a_0\leq b_0)\circ M(a{\overset{}{\rightsquigarrow}}a_0)
        \end{aligned}
    \end{equation}   
\end{proof}

Theorem~\ref{thm-sim-1} is a direct consequence of Proposition~\ref{prop-sim-2}.

\begin{proposition}\label{prop-sim-2}
Let $M:(\mathbb{R},\leq)^{2}\rightarrow \cat{vect}_\mathbb{F}$ be a persistence module. Let $(J_1,\leq)$ and $(J_2,\leq)$ be connected and convex subposets of $(\mathbb{R},\leq)^{2}$ such that $J_1\cap J_2=\emptyset$ and $M(i\leq j): Mi\rightarrow Mj$ is an isomorphism for all $i\leq j\in J_k$, $k=1,2$. If there exists $a\in J_1$ and $b\in J_2$ such that $a\leq b$, then there exists a well-defined morphism $\phi_{J_1J_2}:\lim_{J_1}M\rightarrow \lim_{J_2}M$ such that the following diagram commutes:
$$
\begin{tikzcd}
\lim_{J_1}M\arrow[r, dashed, "\phi_{J_1J_2}"]\arrow[d, "\sigma_a"']
&\lim_{J_2}M\arrow[d, "\tau_b"]\\
Ma \arrow[r, "M(a\leq b)"'] 
    & Mb 
\end{tikzcd}
$$
where $\sigma_i$ is the leg map of the limit cone $\lim_{J_1}M\Rightarrow\restr{M}{J_1}$ and $\tau_j$ is the leg map of the limit cone $\lim_{J_2}M\Rightarrow\restr{M}{J_2}$.
\end{proposition}

\begin{proof}
We use the same construction as in Proposition~\ref{prop-sim-1}, then both $\sigma_i$ and $\tau_j$ are isomorphisms for all $i\in J_1$ and $j\in J_2$ because they are compositions of isomorphisms. Define $\phi_{J_1J_2}=\tau_b^{-1}\circ M(i\leq j)\circ \sigma_a$. Now we will show $\phi_{J_1J_2}$ does not depend on the choice of $\sigma_a$ and $\tau_b$.

Assume there exists $a'\leq b'$ where $a'\in J_1$ and $b'\in J_2$. Let $\sigma_{a'}: \lim_{J_1}M\rightarrow Ma'$ be the leg map of the limit cone and $\tau_{b'}: \lim_{J_2}M\rightarrow Mb'$ be the leg map of the limit cone. Since $J_1$ is connected, there exists a zigzag path (denoted by $a\rightsquigarrow a'$) from $a$ to $a'$. Similarly, since $J_2$ is connected, there exists a zigzag path (denoted by $b\rightsquigarrow b'$) from $b$ to $b'$. Lemma~\ref{prop-lemma-1} implies that $M(a\rightsquigarrow a')\circ \sigma_a=\sigma_a'$ and $M(b\rightsquigarrow b')\circ \tau_b=\tau_b'$. 

Note that 
\begin{equation}
\begin{aligned}
  \phi_{J_1J_2}&=\tau_b^{-1}\circ M(a\leq b)\circ \sigma_{a}\\
  &=\tau_b'^{-1}\circ M(b\rightsquigarrow b')\circ M(a\leq b)\circ M(a\rightsquigarrow a')^{-1}\circ \sigma_{a'}\\
  &=\tau_b'^{-1}\circ M(a\rightsquigarrow b')\circ M(a\rightsquigarrow a')^{-1}\circ \sigma_{a'}\\
  &=\tau_b'^{-1}\circ M(a'\leq b')\circ \sigma_{a'}
\end{aligned}
\end{equation}

Hence, $\phi_{J_1J_2}$ does not depend on the choice of $a$ and $b$.
\end{proof}

The following example shows that the assumption of connectedness of the chambers is necessary.

\begin{example}
    Let $(P,\leq)$ be a subposet of $(\mathbb{R},\leq)^2$ where $$P=\{(x,0)\mid x\leq 0\}\cup \{(0,y)\mid y\leq 0\}$$
    Let $M\in\cat{vect}_{\mathbb{F}}^{(\mathbb{R},\leq)^{2}}$ such that $$M((x,y)) = \begin{cases} 0, & \mbox{if } (x,y)\notin P\\ \mathbb{F}^2, & \mbox{if } (x,y)=(0,0)\\ \mathbb{F}, & \mbox{if } (x,y)\in P-\{(0,0)\}\end{cases}$$
    and 
$$M((x,y)\leq (x',y')) = \begin{cases} 
\id_{\mathbb{F}}, & \mbox{if } (x,y),(x',y')\in P-\{(0,0)\}\\
\begin{bmatrix}
1\\
0
\end{bmatrix}, & \mbox{if } (x,y)\in \{(x,0)\mid x\leq 0\}, (x,y')=(0,0)\\
\begin{bmatrix}
0\\
1
\end{bmatrix}, & \mbox{if } (x,y)\in \{(0,y)\mid y\leq 0\}, (x',y)=(0,0)\\
 0, & \mbox{else }\end{cases}$$
 
The support of $M$ has two isotopy regions: $P-\{(0,0)\}$ and $\{(0,0)\}$, where $P-\{(0,0)\}$ has two connected components. However, there is no canonical morphism $\mathbb{F}\rightarrow \mathbb{F}^2$ induced by the morphisms in $M$. In fact, the subdivision of $P$ into $P-\{(0,0)\}$ and $\{(0,0)\}$ is not a constant subdivision of $P$ subordinate to $M$.
\end{example}

Lemma~\ref{lemma:no-inverse} guarantees that there is no morphism $\phi_{J_2J_1}$ induced by the morphisms of $M$ under the assumption of Proposition~\ref{prop-sim-2}. The following example~\cite{miller2020homological} shows that the convexity assumption of the chambers is necessary.

\begin{example}
    Let $M\in\cat{vect}_{\mathbb{F}}^{(\mathbb{R},\leq)^{2}}$ such that $$M((x,y)) = \begin{cases} \mathbb{F}^2, & \mbox{if } (x,y)=(0,0)\\ \mathbb{F}, & \mbox{if } (x,y)\neq(0,0)\end{cases}$$
    and 
$$M((x,y)\leq (x',y')) = \begin{cases} 
\id_{\mathbb{F}}, & \mbox{if } (x,y),(x',y')\in \mathbb{R}^2-\{(0,0)\}\\
\begin{bmatrix}
1\\
0
\end{bmatrix}, & \mbox{if } (x',y')=(0,0)\\
\begin{bmatrix}
1 &1
\end{bmatrix}, & \mbox{if } (x,y)=(0,0)\\
 \id_{\mathbb{F}^2}, & \mbox{if } (x,y)=(x',y')=(0,0)\end{cases}$$

Note that $M$ has two isotopy regions: $\mathbb{R}^2-\{(0,0)\}$ and $\{(0,0)\}$, where $\mathbb{R}^2-\{(0,0)\}$ is not convex. The isotopy subdivision is also a constant subdivision of $(\mathbb{R},\leq)^2$ subordinate to $M$. However, there are canonical morphisms $\mathbb{F}\rightarrow \mathbb{F}^2$ and $\mathbb{F}^2\rightarrow \mathbb{F}$ induced by the morphisms in $M$.
\end{example}





\section{Finite Encodings of the $2$-Parameter p.f.d. Persistence Modules with Finite Convex Isotopy Subdivisions}
Let $M\in\cat{vect}_{\mathbb{F}}^{(\mathbb{R},\leq)^{2}}$. In this section, we show that there exists a finite encoding of $M$ if there exists a finite convex isotopy subdivision of  $(\mathbb{R},\leq)^{2}$ subordinate to $M$.

Construct a set (denoted by $P$) and a homogeneous binary relation (denoted by $\rightarrow$) on $P$ as follows:

\begin{itemize}
\item Each chamber $J_p$ in $(\mathbb{R},\leq)^{2}$ corresponds to an element $p\in P$;
\item Given $p,q\in P$. $p\rightarrow q$ if 
\begin{enumerate}
    \item $p=q$ or 
    \item $p\neq q$, and there exists $x\in J_p$ and $y\in J_q$ such that $x\leq y$.
\end{enumerate}
\end{itemize}

\begin{lemma}
   $(P,\leq)$ is a poset, where $\leq$ is the transitive closure of $\rightarrow$. 
\end{lemma}
\begin{proof}
    We only need to show that $\rightarrow$ is antisymmetric. Assume $p\rightarrow q$ and $q\rightarrow p$, but $p\neq q$. Then $J_p\cap J_q=\emptyset$. If there exists $x,x'\in J_p$ and $y,y'\in J_q$ such that $x\leq y$ and $x'\geq y'$, Lemma~\ref{lemma:no-inverse} implies that $J_p\cap J_q\neq \emptyset$. Contradiction. Hence $p=q$.
\end{proof}

For example, the poset $(P,\leq)$ associated to the hyperplane arrangement of $Y^2_{r,L_{e_1}}$ is shown in Figure~\ref{ch4.1:poset}, where $Y^2_{r,L_{e_1}}$ is the second configuration space of the $\mathsf{Star_3}$ graph with restraint parameter $r$ and edge length vector $L=(L_{e_1},1,1)$. 

\begin{figure}[hbtp!]
\centering
\resizebox{0.45\textwidth}{!}{

\tikzset{every picture/.style={line width=0.75pt}} 

\begin{tikzpicture}[x=0.75pt,y=0.75pt,yscale=-1,xscale=1]

\draw  (4.3,200.35) -- (192.8,200.35)(23.15,27.1) -- (23.15,219.6) (185.8,195.35) -- (192.8,200.35) -- (185.8,205.35) (18.15,34.1) -- (23.15,27.1) -- (28.15,34.1)  ;
\draw    (22.15,201.35) -- (188.5,35) ;
\draw    (70,36) -- (70,200) ;
\draw    (120,36) -- (120,200) ;
\draw    (70,200) -- (189.5,80.5) ;
\draw  [draw opacity=0][fill={rgb, 255:red, 245; green, 166; blue, 35 }  ,fill opacity=0.4 ] (70.5,36) -- (70,153) -- (22.65,200.35) -- (23.5,36) -- cycle ;
\draw  [draw opacity=0][fill={rgb, 255:red, 155; green, 155; blue, 155 }  ,fill opacity=0.4 ] (120,36) -- (120,104) -- (70,153) -- (70,36) -- cycle ;
\draw  [draw opacity=0][fill={rgb, 255:red, 189; green, 16; blue, 224 }  ,fill opacity=0.4 ] (188.5,36) -- (120,104) -- (120,36) -- cycle ;
\draw  [draw opacity=0][fill={rgb, 255:red, 74; green, 144; blue, 226 }  ,fill opacity=0.4 ] (188.5,36) -- (188.5,82) -- (119.5,150.5) -- (120,104) -- cycle ;
\draw  [draw opacity=0][fill={rgb, 255:red, 248; green, 231; blue, 28 }  ,fill opacity=0.4 ] (120,104) -- (119.5,150.5) -- (70,200) -- (70,153) -- cycle ;
\draw  [draw opacity=0][fill={rgb, 255:red, 126; green, 211; blue, 33 }  ,fill opacity=0.4 ] (70,153) -- (70,200) -- (23.15,200.35) -- cycle ;
\draw  [draw opacity=0][fill={rgb, 255:red, 208; green, 2; blue, 27 }  ,fill opacity=0.4 ] (119.5,150.5) -- (120,200) -- (70,200) -- cycle ;
\draw [color={rgb, 255:red, 74; green, 80; blue, 226 }  ,draw opacity=1 ]   (95.5,56.5) -- (43.5,56.5) ;
\draw [shift={(43.5,56.5)}, rotate = 180] [color={rgb, 255:red, 74; green, 80; blue, 226 }  ,draw opacity=1 ][fill={rgb, 255:red, 74; green, 80; blue, 226 }  ,fill opacity=1 ][line width=0.75]      (0, 0) circle [x radius= 3.35, y radius= 3.35]   ;
\draw [shift={(63.5,56.5)}, rotate = 360] [color={rgb, 255:red, 74; green, 80; blue, 226 }  ,draw opacity=1 ][line width=0.75]    (10.93,-3.29) .. controls (6.95,-1.4) and (3.31,-0.3) .. (0,0) .. controls (3.31,0.3) and (6.95,1.4) .. (10.93,3.29)   ;
\draw [shift={(95.5,56.5)}, rotate = 180] [color={rgb, 255:red, 74; green, 80; blue, 226 }  ,draw opacity=1 ][fill={rgb, 255:red, 74; green, 80; blue, 226 }  ,fill opacity=1 ][line width=0.75]      (0, 0) circle [x radius= 3.35, y radius= 3.35]   ;
\draw [color={rgb, 255:red, 74; green, 80; blue, 226 }  ,draw opacity=1 ]   (147.5,56.5) -- (95.5,56.5) ;
\draw [shift={(95.5,56.5)}, rotate = 180] [color={rgb, 255:red, 74; green, 80; blue, 226 }  ,draw opacity=1 ][fill={rgb, 255:red, 74; green, 80; blue, 226 }  ,fill opacity=1 ][line width=0.75]      (0, 0) circle [x radius= 3.35, y radius= 3.35]   ;
\draw [shift={(115.5,56.5)}, rotate = 360] [color={rgb, 255:red, 74; green, 80; blue, 226 }  ,draw opacity=1 ][line width=0.75]    (10.93,-3.29) .. controls (6.95,-1.4) and (3.31,-0.3) .. (0,0) .. controls (3.31,0.3) and (6.95,1.4) .. (10.93,3.29)   ;
\draw [shift={(147.5,56.5)}, rotate = 180] [color={rgb, 255:red, 74; green, 80; blue, 226 }  ,draw opacity=1 ][fill={rgb, 255:red, 74; green, 80; blue, 226 }  ,fill opacity=1 ][line width=0.75]      (0, 0) circle [x radius= 3.35, y radius= 3.35]   ;
\draw [color={rgb, 255:red, 74; green, 80; blue, 226 }  ,draw opacity=1 ]   (95.5,146.5) -- (43.5,193) ;
\draw [shift={(43.5,193)}, rotate = 138.2] [color={rgb, 255:red, 74; green, 80; blue, 226 }  ,draw opacity=1 ][fill={rgb, 255:red, 74; green, 80; blue, 226 }  ,fill opacity=1 ][line width=0.75]      (0, 0) circle [x radius= 3.35, y radius= 3.35]   ;
\draw [shift={(65.03,173.75)}, rotate = 318.2] [color={rgb, 255:red, 74; green, 80; blue, 226 }  ,draw opacity=1 ][line width=0.75]    (10.93,-3.29) .. controls (6.95,-1.4) and (3.31,-0.3) .. (0,0) .. controls (3.31,0.3) and (6.95,1.4) .. (10.93,3.29)   ;
\draw [shift={(95.5,146.5)}, rotate = 138.2] [color={rgb, 255:red, 74; green, 80; blue, 226 }  ,draw opacity=1 ][fill={rgb, 255:red, 74; green, 80; blue, 226 }  ,fill opacity=1 ][line width=0.75]      (0, 0) circle [x radius= 3.35, y radius= 3.35]   ;
\draw [color={rgb, 255:red, 74; green, 80; blue, 226 }  ,draw opacity=1 ]   (147.5,94.5) -- (95.5,146.5) ;
\draw [shift={(95.5,146.5)}, rotate = 135] [color={rgb, 255:red, 74; green, 80; blue, 226 }  ,draw opacity=1 ][fill={rgb, 255:red, 74; green, 80; blue, 226 }  ,fill opacity=1 ][line width=0.75]      (0, 0) circle [x radius= 3.35, y radius= 3.35]   ;
\draw [shift={(117.26,124.74)}, rotate = 315] [color={rgb, 255:red, 74; green, 80; blue, 226 }  ,draw opacity=1 ][line width=0.75]    (10.93,-3.29) .. controls (6.95,-1.4) and (3.31,-0.3) .. (0,0) .. controls (3.31,0.3) and (6.95,1.4) .. (10.93,3.29)   ;
\draw [shift={(147.5,94.5)}, rotate = 135] [color={rgb, 255:red, 74; green, 80; blue, 226 }  ,draw opacity=1 ][fill={rgb, 255:red, 74; green, 80; blue, 226 }  ,fill opacity=1 ][line width=0.75]      (0, 0) circle [x radius= 3.35, y radius= 3.35]   ;
\draw [color={rgb, 255:red, 74; green, 80; blue, 226 }  ,draw opacity=1 ]   (95.5,146.5) -- (95.5,56.5) ;
\draw [shift={(95.5,56.5)}, rotate = 270] [color={rgb, 255:red, 74; green, 80; blue, 226 }  ,draw opacity=1 ][fill={rgb, 255:red, 74; green, 80; blue, 226 }  ,fill opacity=1 ][line width=0.75]      (0, 0) circle [x radius= 3.35, y radius= 3.35]   ;
\draw [shift={(95.5,95.5)}, rotate = 90] [color={rgb, 255:red, 74; green, 80; blue, 226 }  ,draw opacity=1 ][line width=0.75]    (10.93,-3.29) .. controls (6.95,-1.4) and (3.31,-0.3) .. (0,0) .. controls (3.31,0.3) and (6.95,1.4) .. (10.93,3.29)   ;
\draw [shift={(95.5,146.5)}, rotate = 270] [color={rgb, 255:red, 74; green, 80; blue, 226 }  ,draw opacity=1 ][fill={rgb, 255:red, 74; green, 80; blue, 226 }  ,fill opacity=1 ][line width=0.75]      (0, 0) circle [x radius= 3.35, y radius= 3.35]   ;
\draw [color={rgb, 255:red, 74; green, 80; blue, 226 }  ,draw opacity=1 ]   (95.5,192) -- (95.5,146.5) ;
\draw [shift={(95.5,146.5)}, rotate = 270] [color={rgb, 255:red, 74; green, 80; blue, 226 }  ,draw opacity=1 ][fill={rgb, 255:red, 74; green, 80; blue, 226 }  ,fill opacity=1 ][line width=0.75]      (0, 0) circle [x radius= 3.35, y radius= 3.35]   ;
\draw [shift={(95.5,163.25)}, rotate = 90] [color={rgb, 255:red, 74; green, 80; blue, 226 }  ,draw opacity=1 ][line width=0.75]    (10.93,-3.29) .. controls (6.95,-1.4) and (3.31,-0.3) .. (0,0) .. controls (3.31,0.3) and (6.95,1.4) .. (10.93,3.29)   ;
\draw [shift={(95.5,192)}, rotate = 270] [color={rgb, 255:red, 74; green, 80; blue, 226 }  ,draw opacity=1 ][fill={rgb, 255:red, 74; green, 80; blue, 226 }  ,fill opacity=1 ][line width=0.75]      (0, 0) circle [x radius= 3.35, y radius= 3.35]   ;
\draw [color={rgb, 255:red, 74; green, 80; blue, 226 }  ,draw opacity=1 ]   (147.5,94.5) -- (147.5,56.5) ;
\draw [shift={(147.5,56.5)}, rotate = 270] [color={rgb, 255:red, 74; green, 80; blue, 226 }  ,draw opacity=1 ][fill={rgb, 255:red, 74; green, 80; blue, 226 }  ,fill opacity=1 ][line width=0.75]      (0, 0) circle [x radius= 3.35, y radius= 3.35]   ;
\draw [shift={(147.5,69.5)}, rotate = 90] [color={rgb, 255:red, 74; green, 80; blue, 226 }  ,draw opacity=1 ][line width=0.75]    (10.93,-3.29) .. controls (6.95,-1.4) and (3.31,-0.3) .. (0,0) .. controls (3.31,0.3) and (6.95,1.4) .. (10.93,3.29)   ;
\draw [shift={(147.5,94.5)}, rotate = 270] [color={rgb, 255:red, 74; green, 80; blue, 226 }  ,draw opacity=1 ][fill={rgb, 255:red, 74; green, 80; blue, 226 }  ,fill opacity=1 ][line width=0.75]      (0, 0) circle [x radius= 3.35, y radius= 3.35]   ;
\draw [color={rgb, 255:red, 74; green, 80; blue, 226 }  ,draw opacity=1 ]   (43.5,193) -- (43.5,56.5) ;
\draw [shift={(43.5,56.5)}, rotate = 270] [color={rgb, 255:red, 74; green, 80; blue, 226 }  ,draw opacity=1 ][fill={rgb, 255:red, 74; green, 80; blue, 226 }  ,fill opacity=1 ][line width=0.75]      (0, 0) circle [x radius= 3.35, y radius= 3.35]   ;
\draw [shift={(43.5,118.75)}, rotate = 90] [color={rgb, 255:red, 74; green, 80; blue, 226 }  ,draw opacity=1 ][line width=0.75]    (10.93,-3.29) .. controls (6.95,-1.4) and (3.31,-0.3) .. (0,0) .. controls (3.31,0.3) and (6.95,1.4) .. (10.93,3.29)   ;
\draw [shift={(43.5,193)}, rotate = 270] [color={rgb, 255:red, 74; green, 80; blue, 226 }  ,draw opacity=1 ][fill={rgb, 255:red, 74; green, 80; blue, 226 }  ,fill opacity=1 ][line width=0.75]      (0, 0) circle [x radius= 3.35, y radius= 3.35]   ;

\draw (197,197.4) node [anchor=north west][inner sep=0.75pt]    {$r$};
\draw (3,15.4) node [anchor=north west][inner sep=0.75pt]    {$L_{e_1}$};
\draw (65,203.4) node [anchor=north west][inner sep=0.75pt]    {$1$};
\draw (114,203.4) node [anchor=north west][inner sep=0.75pt]    {$2$};

\end{tikzpicture}
}
\caption{The Hasse diagram of $(P,\leq)$ associated to the hyperplane arrangement in the parameter space of $PH_0(Y^2_{r,L_{e_1}};\mathbb{F})$~\cite{li2023persistent}}
\label{ch4.1:poset}
\end{figure}
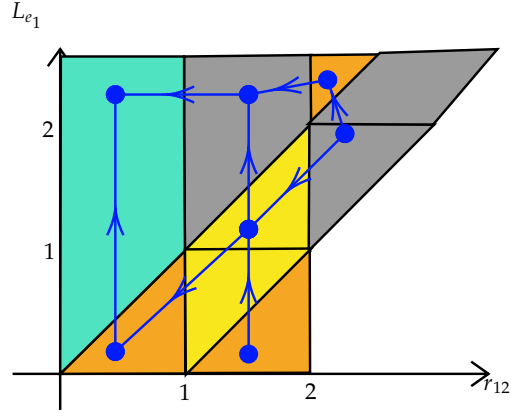

\begin{lemma}
  The construction of $(P,\leq)$ is functorial, i.e., the construction above gives a functor (denoted by $\mathcal{F}_M$) $$\mathcal{F}_M:(\mathbb{R},\leq)^{2}\rightarrow (P,\leq) $$   
\end{lemma}
\begin{proof}
At the object level, $\mathcal{F}_M$ sends $(x,y)\in (\mathbb{R},\leq)^{2}$ to $p\in P$, where $J_p$ is the chamber that contains $(x,y)$.

At the morphism level, $\mathcal{F}_M$ sends $(x,y)\rightarrow (x',y')$ to the unique morphism $p\rightarrow p'$, where $J_p$ represents the chamber that contains $(x,y)$ and $J_p'$ represents the chamber that contains $(x',y')$. Lemma~\ref{lemma:no-inverse} implies that $\mathcal{F}_M$ is well-defined because the chambers in $(\mathbb{R},\leq)^{2}$ are convex, connected, and pairwise disjoint.

Note that when $p=p'$, $(x,y)$ and $(x',y')$ are in the same chamber. Without loss of generality, we assume $(x,y)\leq (x',y')$. Note that $\mathcal{F}_M$ sends $(x,y)\leq (x',y')$ to $\id_p$. Hence, in particular, 
$$\mathcal{F}_M\id_{(x,y)}=\id_p$$

Since $(P,\leq)$ is thin, for all $(x,y), (x',y'),(x'',y'')\in (\mathbb{R},\leq)^{2}$ such that $(x,y)\leq (x',y')\leq (x'',y'')$,
$$\mathcal{F}_M((x',y')\leq (x'',y''))\circ \mathcal{F}_M((x,y)\leq (x',y'))=\mathcal{F}_M((x,y)\leq (x'',y''))$$

In conclusion, $\mathcal{F}_M$ is a functor. 
\end{proof}

Given $M\in\cat{vect}_{\mathbb{F}}^{(\mathbb{R},\leq)^{2} }$. Assume $M$ admits a finite convex isotopy subdivision of $(\mathbb{R},\leq)^{2}$ and let $(P,\leq)$ be the poset of chambers. Define $\widetilde{N}:(P,\leq)\rightarrow \cat{vect}_{\mathbb{F}}$ as follows:
\begin{itemize}
    \item Given $p\in P$, $\widetilde{N}_p=\lim_{J_p}M\cong Mi$ for all $i\in J_p$;
    \item For all $p\leq q\in P$, $$\widetilde{N}(p\leq q)=\begin{cases} \id_{\lim_{J_p}M}, & \mbox{if } p=q \\ \phi_{J_pJ_q} , & \mbox{else } \end{cases}$$
where $\phi_{J_pJ_q}$ is the canonical linear transformation defined in Proposition~\ref{prop-sim-2}.
\end{itemize} 

\begin{proof}[Proof of Theorem~\ref{thm-sim-1'}]
Note that $\widetilde{N}\circ \mathcal{F}_M=N$, where $N$ is constructed in Theorem~\ref{thm-sim-1}. Moreover, Theorem~\ref{thm-sim-1} implies $M\cong N$. Hence $M\cong \widetilde{N}\circ \mathcal{F}_M$.
\end{proof}

\begin{theorem}
Let $M\in\cat{vect}_{\mathbb{F}}^{(\mathbb{R},\leq)^{2}}$. Assume there exists a finite convex isotopy subdivision of $(\mathbb{R},\leq)^{2}$ subordinate to $M$. Then $\mathcal{F}_M^\ast: \cat{vect}_{\mathbb{F}}^{(P,\leq)}\rightarrow\cat{vect}_{\mathbb{F}}^{(\mathbb{R},\leq)^{2}}$ is fully-faithful, where $(P,\leq)$ is the poset of the chambers in $(\mathbb{R},\leq)^{2}$.
\end{theorem}

\begin{proof}
    Let $\{J_p\mid p\in (P,\leq)\}$ be the set of chambers of the finite partition of $(\mathbb{R},\leq)^{2}$ that satisfyies the conditions of Theorem~\ref{thm-sim-1}. Consider $\Phi:\hom(N,N')\rightarrow \hom(N\circ \mathcal{F}_M,N'\circ \mathcal{F}_M)$ where $\Phi((\alpha_p)_{p\in (P,\leq)})=(\beta_i)_{i\in (\mathbb{R},\leq)^{2}}$ and $\beta_i=\alpha_p$ when $i\in J_p$. It is clear that $\Phi$ is injective.

    Let $(\beta_i)_{i\in (\mathbb{R},\leq)^{2}}\in \hom(N\circ \mathcal{F}_M,N'\circ \mathcal{F}_M)$. Given $p\in (P,\leq)$, the universal property of the limit implies that there exists a unique morphism $\alpha_p:\lim_{J_p}M\rightarrow \lim_{J_p}M'$. Assume $p\leq q\in (P,\leq)$, the universal property of the limit implies that there exists a unique morphism $\lim_{J_p}M\rightarrow \lim_{J_q}M'$. Hence $(\alpha_p)_{p\in(P,\leq)}$ is a morphism in  $\cat{vect}_{\mathbb{F}}^{(P,\leq)}$. Therefore, $\Phi$ is surjective. 

    In conclusion, $\mathcal{F}_M^\ast$ is fully-faithful.
\end{proof}

 
\section{Thin Polycode Modules}\label{TPM}
We have seen in Lemma~\ref{thinpoly} that polytope modules are thin and indecomposable. In this section, we will show that the converse of Lemma~\ref{thinpoly} is also true (up to isomorphism) when the indexing poset of the persistence module is $(\mathbb{R},\leq)^2$.

\begin{lemma}\label{chap3.3:lemma1}
Every indecomposable thin persistence module $M\in\cat{vect}_{\mathbb{F}}^{(\mathbb{R}, \leq)^2}$ has connected support. Moreover, the support is a subposet of $(\mathbb{R}, \leq)^2$.
\end{lemma}
\begin{proof}
Let $P=(P_0,P_1)$ be the support\footnote{Given a quiver representation $M$ of a quiver $Q=(Q_0,Q_1)$, the support quiver $P$ of $M$ consists of the following data: its vertices are the vertices $a\in Q_0$ such that $Ma\neq 0$, and the arrows of $P$ are the arrows $M(a\rightarrow a')\neq 0$ where $a\rightarrow a'\in Q_1$.} of $M$. Define a binary relation $\leq$ on $P_0$ as follows: given $a,b\in P_0$, $a\leq b$ iff $a\leq b$ in $(\mathbb{R}, \leq)^2$ and $M(a\leq b)\neq 0$. Given $a\in P_0$, $M(a\leq a)=\id_{Ma}$. Moreover, given $a,b,c\in P_0$ where $M(a\leq b)\neq 0$ and $M(b\leq c)\neq 0$, note that $M(a\leq c)=M(a\leq b)\circ M(b\leq c)\neq 0$ because $M$ is a thin representation. Therefore, $\leq$ is a partial order on $P_0$. We denote this poset $P_0$ with the partial order $\leq$ as $(P,\leq)$. It is clear that $(P,\leq)$ is a subposet of $(\mathbb{R}, \leq)^2$.

We now claim that $(P,\leq)$ is connected: Assume $(P,\leq)$ is not connected, then there exists $a,b\in P$ such that there is no zigzag path connecting $a$ and $b$. Define $$S=\{x\in P: \mbox{$x$ and $a$ is joined by a zigzag path in $(P,\leq)$ }\}$$ and $T=S^c$. Define $N,N':(\mathbb{R},\leq)^2\rightarrow\cat{vect}_{\mathbb{F}}$ by

$$Ni=\begin{cases} Mi, & \mbox{if } i\in S \\ 0, & \mbox{else }  \end{cases}, \quad N(i\leq j)=\begin{cases} M(i\leq j), & \mbox{if } i\leq j\in S \\ 0, & \mbox{else }  \end{cases}$$
and $$N'i=\begin{cases} Mi, & \mbox{if } i\in T \\ 0, & \mbox{else }  \end{cases}, \quad N'(i\leq j)=\begin{cases} M(i\leq j), & \mbox{if } i\leq j\in T \\ 0, & \mbox{else }  \end{cases}$$

It is clear that $M\cong N\oplus N'$, and both $N$ and $N'$ are not trivial since $Na\neq 0$ and $N'b\neq 0$. Contradiction. Therefore, $(P,\leq)$ is connected.
\end{proof}

\begin{lemma}\label{chap3.3:lemma2}
Let $M\in\cat{Vect}_{\mathbb{F}}^{(\mathbb{R}, \leq)^2}$ be a thin persistence module and let $(P,\leq)$ be the support of $M$. Then for any $a,b\in P$ such that $a<b$, $M(a<b): Ma\rightarrow Mb$ is not $0$ if there exists a zigzag path from $a$ to $b$ in $(P,\leq)$.
\end{lemma}
\begin{proof}
Since $(\mathbb{R}, \leq)^2$ is a thin category, we can assume such a zigzag path in $(P,\leq)$ consists of horizontal and vertical arrows. We say an arrow on the zigzag path is a good arrow if the orientation of the arrow coincides with the orientation of the zigzag path; otherwise, we say the arrow is a bad arrow. WLOG, we may assume all zigzag paths are reduced: the zigzag path doesn't have two (or more) consecutive horizontal or vertical arrows, i.e., we either
\begin{itemize}
\item combine the two (or more) consecutive horizontal or vertical arrows if they are all good arrows or all bad arrows;
\item get a new (shorter) arrow from combining a good horizontal (vertical, resp) arrow with a bad horizontal (vertical, resp) arrow. 
\begin{itemize}
\item The new arrow is good if the length of the original good arrow is strictly greater than the length of the original bad arrow;
\item The new arrow is bad if the length of the original good arrow is strictly less than the length of the original bad arrow;
\item The new arrow is a vertex if the length of the original good arrow is equal to the length of the original bad arrow.
\end{itemize}

\end{itemize}
Induction on the length of the zigzag path.
\begin{itemize}
\item length$=1$. There is nothing to show.
\item length$=N\rightarrow N+1$. 
\begin{itemize}
\item If the zigzag path has at least one self-intersection (denote a self-intersection by $c$) then we can write the zigzag path as follows $$a-x_1-\cdots -x_m-c-y_1-\cdots-y_n-c-z_1-\cdots z_{N-m-n-2}-b$$ where $n\geq 1$. (Note that the induction hypothesis strikes when $c=a$ or $c=b$.) Therefore, $$a-x_1-\cdots -x_m-c-z_1-\cdots z_{N-m-n-2}-b$$ is a zigzag path from $a$ to $b$ with a maximum length of $N$. By induction hypothesis, $M(a<b)\neq 0$.
\item Now we assume the zigzag path has no self-intersection. 
\begin{itemize}
\item If the zigzag path consists of good arrows, then $M(a<b)\neq 0$;
\item If we have a bad (B) arrow on the zigzag path, then 
\begin{itemize}
\item there exist two consecutive good (G) arrows adjacent to the bad (B) arrow, i.e., BGG or GGB;\\
or
\item there exist two consecutive bad (B) arrows adjacent to the good (G) arrow, i.e., GBB or BBG
\end{itemize}
Otherwise, 
\begin{itemize}
\item the zigzag path consists of bad arrows;\\
or
\item good arrow and bad arrow alternate on the zigzag path. 
\end{itemize}
contradicting the assumption that $a<b$. 
\end{itemize}

Because $(\mathbb{R}^2,\leq)$ is thin, we can substitute BBG/GBB/GGB/BGG with two new arrows (maybe degenerate). Therefore, the length of the new zigzag path is at most $N$. By the induction hypothesis, $M(a<b)\neq 0$.

\end{itemize}
\end{itemize}
In conclusion, if there exists a zigzag path from $a<b$ in $(P,\leq)$, then $M(a<b)\neq 0$.
\end{proof}

\begin{proof}[Proof of Theorem\ref{chap3.3:thm}]
Let $(P,\leq)$ be the support of $M$. Lemma~\ref{chap3.3:lemma1} implies that $(P,\leq)$ is connected. Moreover, Lemma~\ref{chap3.3:lemma2} implies that $(P,\leq)$ is convex: otherwise, there exists $a\leq c\leq b\in (\mathbb{R},\leq)^2$ such that $a,b\in P$ and $c\notin P$. Hence $M(a\leq b)=M(c\leq b)\circ M(a\leq c)=0$ because $Mc=0$. Contradiction. Hence $(P,\leq)$ is convex. 

Now we construct the morphism between $M$ and $\mathbb{F}P$. Fix $a\in P$, define $\alpha: M\Rightarrow \mathbb{F}P$ as follows: for all $b\in P$ and $x\in Mb$, 
$$\alpha_{b}(x) = \begin{cases}M(a{\overset{}{\rightsquigarrow}}b)x, & \mbox{if } b\in P \\ 0, & \mbox{else } \end{cases}$$
Lemma~\ref{prop-lemma-1} ensures $\alpha_b$ is well-defined. It is clear that $\alpha$ is a natural transformation, and it is a natural isomorphism because $\alpha_{b}$ is invertible for all $b\in P$, where $$\alpha_{b}^{-1}(x)=M(b{\overset{}{\rightsquigarrow}}a)x=M(a{\overset{}{\rightsquigarrow}}b)^{-1}x$$
\end{proof}


The following example shows that both Theorem \ref{thm-sim-1'} and Theorem \ref{chap3.3:thm} may fail when a persistence module has more than two parameters.

\begin{example}\label{ex-dim3} 
Define $Q=\{(x,y,z)\in\{0,1\}^3\mid 1\leq x+y+z\leq 2\}\subseteq \mathbb{Z}^3$. Equip $Q$ with the partial order induced by the product order on $(\mathbb{Z},\leq)^3$. Given $m\in\mathbb{R}-\{0\}$, consider a thin persistence module $M:(\mathbb{Z},\leq)^3\rightarrow \cat{vect}_{\mathbb{F}}$ where $Mi=\begin{cases} \mathbb{F}, & \mbox{if } i\in Q \\ 0, & \mbox{else } \end{cases}$ and $M(i\leq j)(x)=\begin{cases}  0, & \mbox{if } i\notin Q \mbox{ or } j\notin Q\\
mx, & \mbox{if } i=(0,1,0)\mbox{ and } j=(1,1,0)\\x, & \mbox{else }\end{cases}$. It is clear that $M$ is thin and indecomposable for all $m\neq 0$, and M is isomorphic to the polytope module $\mathbb{F}Q$ if and only if $m=1$.

Since $Mi\cong Mj$ for all $i,j\in Q$ and $Q$ is connected, $M$ has only one (non-trivial) chamber. Therefore, the poset of chambers $(P,\leq)$ is a singleton set. Fix $m\neq 1$. Assume there exists $N\in \cat{vect}_{\mathbb{F}}^{(P,\leq)}$ such that $M\cong N\circ \mathcal{F}_M$. Note that $Np=\mathbb{F}$ for all $p\in P$ and $P=\{p\}$; therefore, $M$ is isomorphic to the polytope module $\mathbb{F}Q$, which is a contradiction.
\end{example} 
\backmatter





\bibliography{app}%

\end{document}